\documentclass[12pt,leqno]{amsart}
\usepackage{amsmath}
\usepackage{amssymb}
\def\overset#1#2{{\mathrel{\mathop {{#2}_{}}\limits^{#1}}}}
\def\underset#1#2{{\mathrel{\mathop {{}_{} {#2}}\limits_{{#1}_{}}}}}
\def\upplim_#1{\underset{#1}{\overline\lim}\;}
\def\lowlim_#1{\underset{#1}{\underline\lim}\;}
\newtheorem{corollary}[equation]{Corollary}
\newtheorem{definition}[equation]{\indent{\it Definition}\rm }

\newtheorem{lemma}[equation]{Lemma}
\newtheorem{proposition}[equation]{Proposition}

\newtheorem{theorem}[equation]{Theorem}

\newcommand{\B}{\mathbb{B}}
\newcommand{\C}{{\mathbb{C}}}

\renewcommand{\P}{{\mathbb{P}}}

\newcommand{\R}{{\mathbb{R}}}
\newcommand{\ddc}{\mathrm{dd^c}}
\newcommand{\rank}{\mathrm{rank}}

\newcommand{\ric}{\mathrm{Ric}}

\numberwithin{equation}{section}

\title[Non-integrated defect relation for meromorphic maps]{Non-integrated defect relation for meromorphic maps from K\"{a}hler manifolds with hypersurfaces of a projective variety in subgeneral position} 

\author{Si Duc Quang}
\author{Quynh Ngoc Le}
\author{Nguyen Thi Nhung}

\begin{document}

\begin{abstract}
In this article, we establish a truncated non-integrated defect relation for meromorphic mappings from a complete K\"{a}hler manifold into a projective variety intersecting a family of hypersurfaces located in subgeneral position, where the truncation level of the defect is explicitly estimated. Our result generalizes and improves previous results. In particular, when the family of hypersurfaces located in general position, our result will implies the previous result of Min Ru-Sogome. In the last part of this paper we will apply our result to study the distribution of the Gauss maps of minimal surfaces.
\end{abstract}

\maketitle

\def\thefootnote{\empty}
\footnotetext{
This research is funded by Vietnam National Foundation for Science and Technology Development (NAFOSTED) under grant number 101.04-2018.01.}

\section{Introduction and Main result} 

Let $M$ be a complete K\"{a}hler manifold of dimension $m$ and let $V$ be a subvariety of $k$-dimension of $\P^n(\C)$. Let $f: M\longrightarrow V$ be a meromorphic mapping and $\Omega_f$ be the pull-back of the Fubini-Study form $\Omega$ on $\P^n(\C)$ by $f$. For a positive integer  $\mu_0$ and a hypersurface $D$ of degree $d$ in $\P^n(\C)$ with $f(M)\not\subset D$, we denote by $\nu_f(D)(p)$ the intersection multiplicity of the image of $f$ and $D$ at $f(p)$. 

In 1985, H. Fujimoto \cite{F85} defined the non-integrated defect of $f$ with respect to $D$ truncated to level $\mu_0$ by
$$\delta_f^{[\mu_0]}:= 1- \inf\{\eta\ge 0: \eta \text{ satisfies condition }(*)\}.$$
Here, the condition (*) means that there exists a bounded non-negative continuous function $h$ on $M$ whose order of each zero is not less than $\min \{\nu_f(D), \mu_0\}$ such that
$$d\eta\Omega_f +\dfrac{\sqrt{-1}}{2\pi}\partial\bar\partial\log h^2\ge [\min\{\nu_f(D), \mu_0\}].$$
After that, Fujimoto gave a result analogous to the defect relation in Nevanlinna theory as follows. 

\vskip0.2cm
\noindent
{\bf Theorem  A} (see \cite[Theorem 1.1]{F85}). \ {\it Let $M$ be an $m$-dimensional complete K\"ahler manifold  and $\omega$ be a K\"ahler form of $M.$  Assume that the universal covering of $M$ is biholomorphic to a ball in $\mathbb C^m.$ Let $f : M\to \P^n(\C)$ be a meromorphic map which is linearly nondegenerate (i.e., its image is not contained in any hyperplane of $\P^n(\C)$). Let $H_1,\ldots,H_q$ be hyperplanes of $\P^n(\C)$ in general position. Assume that there eixst $\rho\ge 0$ and a bounded continuous function $h\ge 0$ on $M$ such that 
$$\rho \Omega_f +\ddc \log h^2\ge \ric \ \omega.$$
Then we have
$$\sum_{i=1}^{q} \delta_{f}^{[n]} (H_i) \le n+1+ \rho n(n+1).$$}

Recently,  M. Ru-S. Sogome \cite{RS} generalized Theorem A to the case of the meromorphic mapping intersecting a family of hypersurfaces in general position. We note that the result of M. Ru-S. Sogome is a natural generalization of the result of H. Fujimoto. After that, Q. Yan \cite{Y13} extended the result of M. Ru-S. Sogome by considering the case where the family of hypersurfaces is in subgeneral position. However, his result does not imply the one of M. Ru and S. Sogome. In 2017, by using the ``replacing hypersurfaces'' method, which is proposed by S. D. Quang \cite{Q16-1}, S. D. Quang-N. T. Q. Phuong and N. T. Nhung were successful to generalise the result of M. Ru-S. Sogome to the case of hypersurfaces in subgeneral position. Their result is stated as follows.

\vskip0.2cm
\noindent
{\bf Theorem  B} (see \cite[Theorem 1.1]{Q16-2}). {\it Let $M$ be an $m$-dimensional complete K\"ahler manifold  and $\omega$ be a K\"ahler form of $M.$  Assume that the universal covering of $M$ is biholomorphic to a ball in $\mathbb C^m.$ Let $f$ be an algebraically nondegenerate meromorphic map of $M$ into $\P^n(\C)$. Let $Q_1,\ldots,Q_q$ be hypersurfaces in $\P^n(\C)$ of degree $d_j,$ in $k$-subgeneral position in $\P^n(\C).$  Let $d = l.c.m.\{d_1,\ldots,d_q\}$ (the least common multiple of $\{d_1,\ldots,d_q\}$). Assume that for some $\rho \ge 0,$ there exists a bounded continuous function $h \geq 0$ on $M$ such that
$$\rho\Omega_f + \ddc \log h^2 \geq \ric \ \omega.$$
Then, for each  $\epsilon>0,$ we have
$$\sum_{j=1}^q {\delta}^{[u-1]}_{f} (Q_j) \le p(n+1)+ \epsilon +\dfrac{\rho u(u-1)}{d},$$
where $p=k-n+1,\ u= \bigl ( ^{N +n}_{\ \ n}\bigl )\leq e^{n+2}(dp(n+1)^2 I(\epsilon^{-1}))^n$ and
$N=(n+1)d+p(n+1)^3I(\epsilon^{-1})d.$}

Here, for a real number $x$, we define $I(x):=\min\{a\in\mathbb Z\ ;\ a>x\}$.
 
We also note that, T. V. Tan and V. V. Truong in \cite{TT} also got a non-integrated defect relation for the family of hypersurfaces in subgeneral position. But their definition of ``subgeneral position'' is very special, which has an extra non natural and complicated condition on the intersection of these $q$ hypersurfaces (see Definition 1.1(ii) in \cite{TT}).

The first aim of this paper is to generalise Theorem B to the case of the families of hypersurfaces in subgeneral position with respect to a projective subvariety of $\P^n(\C)$. Before stating our result, we recall the following.

Let $V$ be a subvariety of $\P^n(\C)$ of dimension $k>0$. Let $N\geq n$ and $q\geq N+1.$ Let $Q_1,\ldots,Q_q$ be hypersurfaces in $\mathbb P^n(\mathbb C).$ The hypersurfaces $Q_1,\ldots,Q_q$ are said to be in $N$-subgeneral position with respect to $V$ if 
$$Q_{j_1}\cap\cdots\cap Q_{j_{N+1}}\cap V=\emptyset\text{ for every }1\leq j_1<\cdots<j_{N+1}\leq q.$$
If  $N=n$ then we say that $Q_1,\ldots,Q_q$ are in \textit{general position} w.r.t $V$.

Our main result is stated as follows.
 
\begin{theorem}\label{1.1} 
Let $M$ be an $m$-dimensional complete K\"ahler manifold  and $\omega$ be a K\"ahler form of $M.$  Assume that the universal covering of $M$ is biholomorphic to a ball in $\mathbb C^m.$ Let $f$ be an algebraically nondegenerate meromorphic map of $M$ into a subvariety $V$ of dimension $k$ in $\P^n(\C)$. Let $Q_1,\ldots,Q_q$ be hypersurfaces in $\P^n(\C)$ of degree $d_j,$ in $N$-subgeneral position with respect to $V$.  Let $d = l.c.m.\{d_1,\ldots,d_q\}$ (the least common multiple of $\{d_1,\ldots,d_q\}$). Assume that for some $\rho \ge 0,$ there exists a bounded continuous function $h \geq 0$ on $M$ such that
$$\rho\Omega_f + \ddc \log h^2 \geq \ric \ \omega.$$
Then, for each  $\epsilon>0,$ we have
$$\sum_{j=1}^q {\delta}^{[M_0-1]}_{f} (Q_j) \le p(k+1)+ \epsilon +\dfrac{\rho\epsilon M_0(M_0-1)}{d},$$
where $p=N-k+1,M_0=\left[d^{k^2+k}\deg (V)^{k+1}e^kp^k(2k+4)^kl^k\epsilon^{-k}+1\right]$ and $l=(k+1)q!$.
\end{theorem}

Here, by the notation $[x]$ we denote the biggest integer which does not exceed the real number $x$.

In the case of the family of hypersurfaces is in general position, we get $N=n$. Moreover, since $(n-k+1)(k+1)\le (\frac{n}{2}+1)^2$ for every $1\le k\le n$,  
letting $\epsilon=1,$
we obtain the following corollary.
\begin{corollary}\label{1.2} Let $f:M\to \P^n(\C)$ be a meromorphic mapping and let $\{Q_i\}_{i=1}^q$ be hypersurfaces in $\P^n(\C)$ of degree $d_i$, located in general position. Then, we have
$$\sum_{j=1}^q {\delta}^{[M_0-1]}_{f} (Q_j) \le \left(\dfrac{n}{2}+1\right)^2+ 1 +\dfrac{\rho M_0(M_0-1)}{d},$$
where $M_0= \left[d^{n^2+n}e^n(2n+4)^nl^n+1\right]$ and $l=(n+1)q!$.
\end{corollary}

In the last part of this paper, we will apply Theorem \ref{1.1} to give a non-integrated defect relation for the Gauss map of a regular submanifold of $\C^m$ (see Theorem \ref{6.1.} below).

\textbf{Acknowledgements.} This work was done during a stay of the first and the second authors at Vietnam Institute for Advanced Study in Mathematics. They would like to thank the institute for their support. 

\section{Basic notions and auxiliary results from Nevanlinna theory}

We recall some notions in Nevanlinna theory for a mapping from a ball in $\C^m$ into $\P^n(\C)$ from \cite{Q16-1,Q16-2} as follows.

\noindent
{\bf 2.1. Counting function.}\ We set $||z|| = \big(|z_1|^2 + \dots + |z_m|^2\big)^{1/2}$ for
$z = (z_1,\dots,z_m) \in \mathbb C^m$ and define
\begin{align*}
\B^m(r) &:= \{ z \in \mathbb C^m : ||z|| < r\},\\
S(r) &:= \{ z \in \mathbb C^m : ||z|| = r\}\ (0<r<\infty).
\end{align*}
Define 
$$v_{m-1}(z) := \big(\ddc ||z||^2\big)^{m-1}\quad \quad \text{and}$$
$$\sigma_m(z):= {\rm d^c} \text{log}||z||^2 \land \big(\ddc \text{log}||z||^2\big)^{m-1}
 \text{on} \quad \mathbb C^m \setminus \{0\}.$$

 For a divisor $\nu$ on  a ball $\B^m(R)$ of $\mathbb C^m$, with $R>0$, and for a positive integer $M$ or $M= \infty$, we define the counting function of $\nu$ by
\begin{align*}
&\nu^{[M]}(z)=\min\ \{M,\nu(z)\},\\
&n(t) =
\begin{cases}
\int\limits_{|\nu|\,\cap \B^m(t)}
\nu(z) v_{m-1} & \text  { if } m \geq 2,\\
\sum\limits_{|z|\leq t} \nu (z) & \text { if }  m=1. 
\end{cases}
\end{align*}
Similarly, we define $n^{[M]}(t).$

Define
$$ N(r,r_0,\nu)=\int\limits_{r_0}^r \dfrac {n(t)}{t^{2m-1}}dt \quad (0<r_0<r<R).$$

Similarly, define  $N(r,r_0,\nu^{[M]})$ and denote it by $N^{[M]}(r,r_0,\nu)$.

Let $\varphi : \B^m(r)\rightarrow \C$ be a meromorphic function. Denote by $\nu_\varphi$ the divisor of zeros of $\varphi$. Define
$$N_{\varphi}(r,r_0)=N(r,r_0,\nu_{\varphi}), \ N_{\varphi}^{[M]}(r,r_0)=N^{[M]}(r,r_0,\nu_{\varphi}).$$

For brevity, we will omit the character $^{[M]}$ if $M=\infty$.

\vskip0.2cm
\noindent
{\bf 2.2. Characteristic function and first main theorem.}\ Let $f : \mathbb \B^m(R) \longrightarrow \mathbb P^n(\mathbb C)$ be a meromorphic mapping.
For arbitrarily fixed homogeneous coordinates $(w_0 : \dots : w_n)$ on $\mathbb P^n(\mathbb C)$, we take a reduced representation $\tilde f = (f_0 , \ldots , f_n)$, which means that each $f_i$ is a holomorphic function on $\B^m(R)$ and  $f(z) = \big(f_0(z) : \dots : f_n(z)\big)$ outside the analytic subset $\{ f_0 = \dots = f_n= 0\}$ of codimension $\geq 2$. Set $\Vert \tilde f \Vert = \big(|f_0|^2 + \dots + |f_n|^2\big)^{1/2}$.

The characteristic function of $f$ is defined by 
$$ T_f(r,r_0)=\int_{r_0}^r\dfrac{dt}{t^{2m-1}}\int\limits_{\B^m(t)}f^*\Omega\wedge v^{m-1}, \ (0<r_0<r<R). $$
By Jensen's formula, we will have
\begin{align*}
T_f(r,r_0)= \int\limits_{S(r)} \log\Vert f \Vert \sigma_m -
\int\limits_{S(r_0)}\log\Vert \tilde f\Vert \sigma_m +O(1), \text{ (as $r\rightarrow R$)}.
\end{align*}

Let $Q$ be a hypersurface in $\P^n(\C)$ of degree $d$. Throughout this paper, we sometimes identify a hypersurface with the defining polynomial if there is no confusion.  Then we may write
$$ Q(\omega)=\sum_{I\in\mathcal T_d}a_I\omega^I, $$
where $\mathcal T_d=\{(i_0,\ldots,i_n)\in\mathbb Z_+^{n+1}\ ;\ i_0+\cdots +i_n=d\}$, $\omega =(\omega_0,\ldots,\omega_n)$, $\omega^I=\omega_0^{i_0}\cdots\omega_n^{i_n}$ with $I=(i_0,\ldots,i_n)\in\mathcal T_d$ and $a_I\ (I\in\mathcal T_d)$ are constants, not all zeros. In the case $d=1$, we call $Q$ a hyperplane of $\P^n(\C)$.

The proximity function of $f$ with respect to $Q$, denoted by $m_f (r,r_0,Q)$, is defined by
$$m_f (r,r_0,Q)=\int_{S(r)}\log\dfrac{||\tilde f||^d}{|Q(\tilde f)|}\sigma_m-\int_{S(r_0)}\log\dfrac{||\tilde f||^d}{|Q(\tilde f)|}\sigma_m,$$
where $Q(\tilde f)=Q(f_0,\ldots,f_n)$. This definition is independent of the choice of the reduced representation of $f$. 

We denote by $f^*Q$ the pullback of the divisor $Q$ by $f$. We may see that $f^*Q$ identifies with the divisor of zeros $\nu^0_{Q(\tilde f)}$ of the function $Q(\tilde f)$. By Jensen's formula, we have
$$N(r,r_0,f^*Q)=N_{Q(\tilde f)}(r,r_0)=\int_{S(r)}\log |Q(\tilde f)|\sigma_m-\int_{S(r_0)}\log |Q(\tilde f)| \sigma_m.$$
Then the first main theorem in Nevanlinna theory for meromorphic mappings and hypersurfaces is stated as follows.

\begin{theorem}[First Main Theorem] Let $f : \mathbb B^m(R) \to \P^n(\C)$ be a holomorphic map, and let $Q$ be a hypersurface in $\P^n(\C)$ of degree $d$. If $f (\mathbb C) \not \subset Q$, then for every real number $r$ with $r_0 < r < R$,
$$dT_f (r,r_0)=m_f (r,r_0,Q) + N(r,r_0,f^*Q)+O(1),$$
where $O(1)$ is a constant independent of $r$.
\end{theorem}

If $ \lim\limits_{r\rightarrow 1}\sup\dfrac{T(r,r_0)}{\log 1/(1-r)}= \infty,$ then the Nevanlinna's defect of $f$ with respect to the hypersurface $Q$ truncated to level $l$ is defined by
$$ \delta^{[l]}_{f,*}(Q)=1-\lim\mathrm{sup}\dfrac{N^{[l]}(r,r_0,f^*Q)}{T_f(r,r_0)}.$$
There is a fact that 
$$0\le \delta^{[l]}_f(Q)\le\delta^{[l]}_{f,*}(Q)\le 1. $$
(See Proposition 2.1 in \cite{RS})

\vskip0.2cm
\noindent
{\bf 2.3. Auxiliary results.}\ Repeating the argument in (\cite{F85}, Proposition 4.5), we have the following.

\begin{proposition}\label{pro2.1}
Let $F_0,\ldots ,F_{N}$ be meromorphic functions on the ball $\B^m(R_0)$ in $\mathbb C^m$ such that $\{F_0,\ldots ,F_{N}\}$ are  linearly independent over $\mathbb C.$ Then  there exists an admissible set  
$$\{\alpha_i=(\alpha_{i1},\ldots,\alpha_{im})\}_{i=0}^{N} \subset \mathbb Z^m_+$$
with $|\alpha_i|=\sum_{j=1}^{m}|\alpha_{ij}|\le i \ (0\le i \le N)$ such that the following are satisfied:

(i)\  $W_{\alpha_0,\ldots ,\alpha_{N}}(F_0,\ldots ,F_{N})\overset{Def}{:=}\det{({\mathcal D}^{\alpha_i}\ F_j)_{0\le i,j\le N}}\not\equiv 0.$ 

(ii) $W_{\alpha_0,\ldots ,\alpha_{N}}(hF_0,\ldots ,hF_{N})=h^{N+1}W_{\alpha_0,\ldots ,\alpha_{N}}(F_0,\ldots ,F_{N})$ for any nonzero meromorphic function $h$ on $\B^m(R_0).$
\end{proposition}

In \cite{RS}, M. Ru and S. Sogome gave the following lemma on logarithmic derivative for the meromorphic mappings of a ball in $\C^m$ into $\P^n(\C)$.

\begin{proposition}[{see \cite{RS}, Proposition 3.3}]\label{pro2.2}
 Let $L_0,\ldots ,L_{N}$ be linear forms of $N+1$ variables and assume that they are linearly independent. Let $F$ be a meromorphic mapping of the ball $\B^m(R_0)\subset\C^m$ into $\P^{N}(\C)$ with a reduced representation $\tilde F=(F_0,\ldots ,F_{N})$ and let $(\alpha_0,\ldots ,\alpha_N)$ be an admissible set of $F$. Set $l=|\alpha_0|+\cdots +|\alpha_N|$ and take $t,p$ with $0< tl< p<1$. Then, for $0 < r_0 < R_0,$ there exists a positive constant $K$ such that for $r_0 < r < R < R_0$,
$$\int_{S(r)}\biggl |z^{\alpha_0+\cdots +\alpha_N}\dfrac{W_{\alpha_0,\ldots ,\alpha_N}(F_0,\ldots ,F_{N})}{L_0(\tilde F)\cdots L_{N}(\tilde F)}\biggl |^t\sigma_m\le K\biggl (\dfrac{R^{2m-1}}{R-r}T_F(R,r_0)\biggl )^p.$$
\end{proposition}
Here $z^{\alpha_i}=z_1^{\alpha_{i1}}\cdots z_m^{\alpha_{im}},$ where $\alpha_i=(\alpha_{i1},\ldots,\alpha_{im})\in\mathbb N^m_0$.

\vskip0.2cm 
\noindent
{\bf 2.4 Chow weights and Hilbert weights.} We recall the notion of Chow weights and Hilbert weights from \cite{R}.

Let $X\subset\P^n(\C)$ be a projective variety of dimension $k$ and degree $\Delta$.  To $X$ we associate, up to a constant scalar, a unique polynomial
$$F_X(\textbf{u}_0,\ldots,\textbf{u}_k) = F_X(u_{00},\ldots,u_{0n};\ldots; u_{k0},\ldots,u_{kn})$$
in $k+1$ blocks of variables $\textbf{u}_i=(u_{i0},\ldots,u_{in}), i = 0,\ldots,k$, which is called the Chow form of $X$, with the following
properties: $F_X$ is irreducible in $\C[u_{00},\ldots,u_{kn}]$, $F_X$ is homogeneous of degree $\Delta$ in each block $\textbf{u}_i, i=0,\ldots,k$, and $F_X(\textbf{u}_0,\ldots,\textbf{u}_k) = 0$ if and only if $X\cap H_{\textbf{u}_0}\cap\cdots\cap H_{\textbf{u}_k}\ne\emptyset$, where $H_{\textbf{u}_i}, i = 0,\ldots,k$, are the hyperplanes given by
$$u_{i0}x_0+\cdots+ u_{in}x_n=0.$$

Let $F_X$ be the Chow form associated to $X$. Let ${\bf c}=(c_0,\ldots, c_n)$ be a tuple of real numbers. Let $t$ be an auxiliary variable. We consider the decomposition
\begin{align*}
\begin{split}
F_X(t^{c_0}u_{00},&\ldots,t^{c_n}u_{0n};\ldots ; t^{c_0}u_{k0},\ldots,t^{c_n}u_{kn})\\ 
& = t^{e_0}G_0(\textbf{u}_0,\ldots,\textbf{u}_n)+\cdots +t^{e_r}G_r(\textbf{u}_0,\ldots, \textbf{u}_n).
\end{split}
\end{align*}
with $G_0,\ldots,G_r\in\C[u_{00},\ldots,u_{0n};\ldots; u_{k0},\ldots,u_{kn}]$ and $e_0>e_1>\cdots>e_r$. The Chow weight of $X$ with respect to $c$ is defined by
\begin{align*}
e_X({\bf c}):=e_0.
\end{align*}
For each subset $J = \{j_0,\ldots,j_k\}$ of $\{0,\ldots,n\}$ with $j_0<j_1<\cdots<j_k,$ we define the bracket
\begin{align*}
[J] = [J]({\bf u}_0,\ldots,{\bf u}_n):= \det (u_{ij_t}), i,t=0,\ldots ,k,
\end{align*}
where $\textbf{u}_i = (u_{i_0},\ldots ,u_{in})$ denotes the blocks of $n+1$ variables. Let $J_1,\ldots ,J_\beta$ with $\beta=\binom{n+1}{k+1}$ be all subsets of $\{0,\ldots,n\}$ of cardinality $k+1$.

Then the Chow form $F_X$ of $X$ can be written as a homogeneous polynomial of degree $\Delta$ in $[J_1],\ldots,[J_\beta]$. We may see that for $\textbf{c}=(c_0,\ldots,c_n)\in\R^{n+1}$ and for any $J$ among $J_1,\ldots,J_\beta$,
\begin{align*}
\begin{split}
[J](t^{c_0}u_{00},\ldots,t^{c_n}u_{0n},&\ldots,t^{c_0}u_{k0},\ldots,t^{c_n}u_{kn})\\
&=t^{\sum_{j\in J}c_j}[J](u_{00},\ldots,u_{0n},\ldots,u_{k0},\ldots,u_{kn}).
\end{split}
\end{align*}

For $\textbf{a} = (a_0,\ldots,a_n)\in\mathbb Z^{n+1}_{\ge 0}$ we write ${\bf x}^{\bf a}$ for the monomial $x^{a_0}_0\cdots x^{a_n}_n$. Let $I=I_X$ be the prime ideal in $\C[x_0,\ldots,x_n]$ defining $X$. Let $\C[x_0,\ldots,x_n]_m$ denote the vector space of homogeneous polynomials in $\C[x_0,\ldots,x_n]$ of degree $m$ (including $0$). For $m = 1, 2,\ldots,$ put $I_m :=\C[x_0,\ldots,x_n]_m\cap I$ and define the Hilbert function $H_X$ of $X$ by
\begin{align*}
H_X(m):=\dim (\C[x_0,\ldots,x_n]_m/I_m).
\end{align*}
By the usual theory of Hilbert polynomials,
\begin{align*}
H_X(m)=\Delta\cdot\dfrac{m^n}{n!}+O(m^{n-1}).
\end{align*}
The $m$-th Hilbert weight $S_X(m,{\bf c})$ of $X$ with respect to the tuple ${\bf c}=(c_0,\ldots,c_n)\in\mathbb R^{n+1}$ is defined by
\begin{align*}
S_X(m,{\bf c}):=\max\left (\sum_{i=1}^{H_X(m)}{\bf a}_i\cdot{\bf c}\right),
\end{align*}
where the maximum is taken over all sets of monomials ${\bf x}^{{\bf a}_1},\ldots,{\bf x}^{{\bf a}_{H_X(m)}}$ whose residue classes modulo $I$ form a basis of $\C[x_0,\ldots,x_n]_m/I_m.$

The following theorem is due to J. Evertse and R. Ferretti \cite{EF01} and is restated by M. Ru \cite{R} for the special case when the field $K=\C$.
\begin{theorem}[{Theorem 4.1 \cite{EF01}, see also Theorem 2.1 \cite{R}}]\label{2.1}
Let $X\subset\P^n(\C)$ be an algebraic variety of dimension $k$ and degree $\Delta$. Let $m>\Delta$ be an integer and let ${\bf c}=(c_0,\ldots,c_n)\in\mathbb R^{n+1}_{\geqslant 0}$.
Then
$$ \dfrac{1}{mH_X(m)}S_X(m,{\bf c})\ge\dfrac{1}{(k+1)\Delta}e_X({\bf c})-\dfrac{(2k+1)\Delta}{m}\cdot\left (\max_{i=0,\ldots,n}c_i\right). $$
\end{theorem}

The following lemma is due to J. Evertse and R. Ferretti \cite{EF02} for the case of the field $\mathbb Q^p$ and reproved by M. Ru \cite{R} for the case of the field $\C$.
\begin{lemma}[{Lemma 3.2 \cite{R}, see also Lemma 5.1 \cite{EF02}}]\label{2.2}
Let $Y$ be a subvariety of $\P^{q-1}(\C)$ of dimension $k$ and degree $\Delta$. Let ${\bf c}=(c_1,\ldots, c_q)$ be a tuple of positive reals. Let $\{i_0,\ldots,i_k\}$ be a subset of $\{1,\ldots,q\}$ such that
$$Y \cap \{y_{i_0}=\cdots =y_{i_k}=0\}=\emptyset.$$
Then
$$e_Y({\bf c})\ge (c_{i_0}+\cdots +c_{i_k})\Delta.$$
\end{lemma}

\section{Non-integrated defect relation for nondegenerate mappings sharing hypersurfaces in subgeneral position}

First of all, we need the following lemmas from \cite{Q16-1,Q16-2}. 
\begin{lemma}[{see Lemma 3.1 \cite{Q16-1} and Lemma 3.1 \cite{Q16-2}}]\label{3.1}
Let $V$ be a projective subvariety of $\P^n(\C)$ of dimension $k$. Let $Q_1,\ldots,Q_{N+1}$ be hypersurfaces in $\P^n(\C)$ of the same degree $d\ge 1$, such that
$$\left (\bigcap_{i=1}^{N+1}Q_i\right )\cap V=\emptyset.$$
Then there exist $k$ hypersurfaces $P_{2},\ldots,P_{k+1}$ of the forms
$$P_t=\sum_{j=2}^{N-k+t}c_{tj}Q_j, \ c_{tj}\in\C,\ t=2,\ldots,k+1,$$
such that $\left (\bigcap_{t=1}^{k+1}P_t\right )\cap V=\emptyset,$ where $P_1=Q_1$.
\end{lemma}

Let $f:M\rightarrow \mathbb P^n(\mathbb C)$ be a meromorphic mapping with a reduced representation $\tilde f=(f_0,\ldots ,f_n)$. We define
$$ Q_i(\tilde f)=\sum_{I\in\mathcal I_d}a_{iI}f^I ,$$
where $f^I=f_0^{i_0}\cdots f_n^{i_n}$ for $I=(i_0,\ldots,i_n)$. Then we can consider $f^*Q_i=\nu_{Q_i(\tilde f)}$ as divisors. We now have the following.

\begin{lemma}[{see Lemma 3.2 \cite{Q16-2}}]\label{3.2}
Let $\{Q_i\}_{i\in R}$ be a family of hypersurfaces in $\mathbb P^n(\mathbb C)$ of the common degree $d$ and let $f$ be a meromorphic mapping of $\mathbb B^m(\C)$ into $\mathbb P^n(\mathbb C)$. Assume that $\bigcap_{i\in R}Q_i=\emptyset$. Then, there exist positive constants $\alpha$ and $\beta$ such that
$$\alpha ||\tilde f||^d \le  \max_{i\in R}|Q_i(\tilde f)|\le \beta ||\tilde f||^d.$$
\end{lemma} 

\begin{lemma}[{see Lemma 3.3 \cite{Q16-2}}]\label{3.3}
Let $\{L_i\}_{i=1}^{u}$ be a family of hypersurfaces in $\mathbb P^n(\mathbb C)$ of the common degree $d$ and let $f$ be a meromorphic mapping of $\B^m(R_0)\subset\mathbb C^m$ into $\mathbb P^n(\mathbb C)$, where $u=\binom{n+d}{n}$. Assume that $\{L_i\}_{i=1}^{u}$ are linearly independent. Then, for every $0<r_0<r<R_0$, we have
$$ T_F(r,r_0)=dT_f(r,r_0)+O(1), $$
where $F$ is the meromorphic mapping of $\B^m(R_0)$ into $\P^{u-1}(\C)$ defined by the representation $F=(L_1(\tilde f):\cdots :L_{u}(\tilde f))$.
\end{lemma} 

\begin{proof}[{\sc Proof of Theorem \ref{1.1}}]
By using the universal covering if necessary, we may assume that $M=\mathbb{B}^m(R_0)$. Replacing $Q_i$ by $Q_i^{d/d_j} \ (j=1,\ldots, q)$ if necessary, we may assume that all hypersurfaces $Q_i\ (1\le i\le q)$ are of the same degree $d$. 
We denote by $\mathcal I$ the set of all permutations of the set $\{1,\ldots.,q\}$. Denote by $n_0$ the cardinality of $\mathcal I$. Then we have $n_0=q!$, and we may write that
$\mathcal I=\{I_1,\ldots,I_{n_0}\}$
where $I_i=(I_i(1),\ldots,I_i(q))\in\mathbb N^q$ and $I_1<I_2<\cdots <I_q$ in the lexicographic order.

For each $I_i\in\mathcal I$, we denote by $P_{i,1},\ldots,P_{i,{k+1}}$ the hypersurfaces obtained in Lemma \ref{3.1} with respect to the hypersurfaces $Q_{I_i(1)},\ldots,Q_{I_i(N+1)}$. It is easy to see that there exists a positive constant $B\ge 1$, which is chosen common for all $I_i\in\mathcal I$, such that
$$ |P_{i,t}(\omega)|\le B\max_{1\le j\le N-k+t}|Q_{I_i(j)}(\omega)|, $$
for all $1\le t\le k+1$ and for all $\omega=(\omega_0,\ldots,\omega_n)\in\C^{n+1}$. 

Consider a reduced representation $\tilde f=(f_0,\ldots ,f_n): \mathbb{B}^m(R_0) \rightarrow \C^{n+1}$ of $f$. Fix an element $I_i\in\mathcal I$. We denote by $S(i)$ the set of all points $z\in \mathbb{B}^m(R_0) \setminus\bigcup_{i=1}^qQ_i(\tilde f)^{-1}(\{0\})$ such that
$$ |Q_{I_i(1)}(\tilde f)(z)|\le |Q_{I_i(2)}(\tilde f)(z)|\le\cdots\le |Q_{I_i(q)}(\tilde f)(z)|.$$
Since $Q_{1},\ldots,Q_{q}$ are in $N$-subgeneral position in $V$, by Lemma \ref{3.2}, there exists a positive constant $A$, which is chosen common for all $I_i$, such that
$$ ||\tilde f (z)||^d\le A\max_{1\le j\le N+1}|Q_{I_i(j)}(\tilde f)(z)|\ \forall z\in S(i). $$
Therefore, for $z\in S(i)$ we have
\begin{align*}
\prod_{i=1}^q\dfrac{||\tilde f (z)||^d}{|Q_i(\tilde f)(z)|}&\le A^{q-N-1}\prod_{j=1}^{N+1}\dfrac{||\tilde f (z)||^d}{|Q_{I_i(j)}(\tilde f)(z)|}\\
&\le A^{q-N-1}B^{k}\dfrac{||\tilde f (z)||^{(N+1)d}}{\bigl (\prod_{j=1}^{N+1-k}|Q_{I_i(j)}(\tilde f)(z)|\bigl )\cdot\prod_{j=2}^{k+1}|P_{i,j}(\tilde f)(z)|}\\
&\le A^{q-N-1}B^{k}\dfrac{||\tilde f (z)||^{(N+1)d}}{|P_{i,1}(\tilde f)(z)|^{N-k+1}\cdot\prod_{j=2}^{k+1}|P_{i,j}(\tilde f)(z)|}\\
&\le A^{q-N-1}B^{k}C^{(N-k)}\dfrac{||\tilde f (z)||^{(N+1)d+(N-k)kd}}{\prod_{j=1}^{k+1}|P_{i,j}(\tilde f)(z)|^{N-k+1}},
\end{align*}
where $C$ is a positive constant, which is chosen common for all $I_i\in\mathcal I$, such that 
$$|P_{i,j}(\omega)|\le C||\omega||^d, \ \forall \omega=(\omega_0,\ldots,\omega_n)\in\C^{n+1}.$$
The above inequality implies that
\begin{align}\label{3.4}
\log \prod_{i=1}^q\dfrac{||\tilde f (z)||^d}{|Q_i(\tilde f)(z)|}\le \log(A^{q-N-1}B^{k}C^{(N-k)})+(N-k+1)\log \dfrac{||\tilde f (z)||^{(k+1)d}}{\prod_{j=1}^{k+1}|P_{i,j}(\tilde f)(z)|}.
\end{align}

We consider the mapping $\Phi$ from $V$ into $\P^{l-1}(\C)\ (l=n_0(k+1))$, which maps a point $x\in V$ into the point $\Phi(x)\in\P^{l-1}(\C)$ given by
$$\Phi(x)=(P_{1,1}(x):\cdots : P_{1,k+1}(x):P_{2,1}(x):\cdots:P_{2,k+1}(x):\cdots:P_{n_0,1}(x):\cdots :P_{n_0,k+1}(x)).$$ 

Let $Y=\Phi (V)$. Since $V\cap\bigcap_{j=1}^{N+1}P_{1,j}=\emptyset$, $\Phi$ is a finite morphism on $V$ and $Y$ is a complex projective subvariety of $\P^{l-1}(\C)$ with $\dim Y=k$ and $\Delta:=\deg Y\le d^k.\deg V$. 
For every 
$${\bf a} = (a_{1,1},\ldots ,a_{1,k+1},a_{2,1}\ldots,a_{2,k+1},\ldots,a_{n_0,1},\ldots,a_{n_0,k+1})\in\mathbb Z^l_{\ge 0}$$ 
and
$${\bf y} = (y_{1,1},\ldots ,y_{1,k+1},y_{2,1}\ldots,y_{2,k+1},\ldots,y_{n_0,1},\ldots,y_{n_0,k+1})$$ 
we denote ${\bf y}^{\bf a} = y_{1,1}^{a_{1,1}}\ldots y_{1,k+1}^{a_{1,k+1}}\ldots y_{n_0,1}^{a_{n_0,1}}\ldots y_{n_0,k+1}^{a_{n_0,k+1}}$. Let $u$ be a positive integer. We set
\begin{align*}
n_u:=H_Y(u)-1,\ l_u:=\binom{l+u-1}{u}-1,
\end{align*}
and define the space
$$ Y_u=\C[y_1,\ldots,y_l]_u/(I_Y)_u, $$
which is a vector space of dimension $n_u+1$. We fix a basis $\{v_0,\ldots, v_{n_u}\}$ of $Y_u$ and consider the meromorphic mapping $F$ with a reduced representation
$$ \tilde F=(v_0(\Phi\circ \tilde f),\ldots ,v_{n_u}(\Phi\circ \tilde f)): \mathbb{B}^m(R_0) \rightarrow \C^{n_u+1}. $$
Hence $F$ is linearly nondegenerate, since $f$ is algebraically nondegenerate.

Then there exists an admissible set $\alpha=(\alpha_0,\ldots,\alpha_{n_u})\in(\mathbb{Z}^m_+)^{n_u+1}$ such that
$$W^\alpha(F_0,\ldots,F_{n_u})=\det (D^{\alpha_i}(v_s(\Phi\circ \tilde f)))_{0\le i,s\le n_u}\neq 0.$$

Now, we fix an index $i\in\{1,\ldots,n_0\}$ and a point $z\in S(i)$. We define 
$${\bf c}_z = (c_{1,1,z},\ldots,c_{1,k+1,z},c_{2,1,z},\ldots,c_{2,k+1,z},\ldots,c_{n_0,1,z},\ldots,c_{n_0,k+1,z})\in\mathbb Z^{l},$$ 
where
\begin{align*}
c_{i,j,z}:=\log\dfrac{||\tilde f(z)||^d||P_{i,j}||}{|P_{i,j}(\tilde f)(z)|}\text{ for } i=1,\ldots,n_0 \text{ and }j=1,\ldots,k+1.
\end{align*}
We see that $c_{i,j,z}\ge 0$ for all $i$ and $j$. By the definition of the Hilbert weight, there are ${\bf a}_{1,z},\ldots,{\bf a}_{H_Y(u),z}\in\mathbb N^{l}$ with
$$ {\bf a}_{i,z}=(a_{i,1,1,z},\ldots,a_{i,1,k+1,z},\ldots,a_{i,n_0,1,z},\ldots,a_{i,n_0,k+1,z}),$$ where $a_{i,j,s,z}\in\{1,\ldots,l_u\}, $
 such that the residue classes modulo $(I_Y)_u$ of ${\bf y}^{{\bf a}_{1,z}},\ldots,{\bf y}^{{\bf a}_{H_Y(u),z}}$ form a basic of $\C[y_1,\ldots,y_l]_u/(I_Y)_u$ and
\begin{align*}
S_Y(u,{\bf c}_z)=\sum_{i=1}^{H_Y(u)}{\bf a}_{i,z}\cdot{\bf c}_z.
\end{align*}
We see that ${\bf y}^{{\bf a}_{i,z}}\in Y_u$ (modulo $(I_Y)_u$). Then we may write
$$ {\bf y}^{{\bf a}_{i,z}}=L_{i,z}(v_0,\ldots ,v_{n_u}), $$ 
where $L_{i,z}\ (1\le i\le H_Y(u))$ are independent linear forms.
We have
\begin{align*}
\log\prod_{i=1}^{H_Y(u)} |L_{i,z}(\tilde F(z))|&=\log\prod_{i=1}^{H_Y(u)}\prod_{\overset{1\le t\le n_0}{1\le j\le k+1}}|P_{t,j}(\tilde f(z))|^{a_{i,t,j,z}}\\
&=-S_Y(u,{\bf c}_z)+duH_Y(u)\log ||\tilde f(z)|| +O(uH_Y(u)).
\end{align*}
Therefore
\begin{align}\label{3.5}
S_Y(u,{\bf c}_z) = \log\prod_{i=1}^{H_Y(u)}\dfrac{1}{|L_{i,z}(\tilde F(z))|} + duH_Y(u)\log ||\tilde f(z)||+O(uH_Y(u)).
\end{align}
From Theorem \ref{2.1} we have
\begin{align}\label{3.6}
\dfrac{1}{uH_Y(u)}S_Y(u,{\bf c}_z)\ge&\dfrac{1}{(k+1)\Delta}e_Y({\bf c}_z)-\dfrac{(2k+1)\Delta}{u}\max_{\underset{1\le i\le n_0}{1\le j\le k+1}}c_{i,j,z}
\end{align}
We chose an index $i_0$ such that $z\in S(i_0)$. It is clear that
\begin{align*}
\max_{\underset{1\le i\le n_0}{1\le j\le k+1}}c_{i,j,z}\le \sum_{1\le j\le k+1}\log\dfrac{||\tilde f(z)||^d||P_{i_0,j}||}{|P_{i_0,j}(\tilde f)(z)|}+O(1),
\end{align*}
where the term $O(1)$ does not depend on $z$ and $i_0$.
Combining (\ref{3.5}), (\ref{3.6}) and the above remark, we get
\begin{align}\nonumber
\dfrac{1}{(k+1)\Delta}e_Y({\bf c}_z)\le & \dfrac{1}{uH_Y(u)}\log\prod_{i=1}^{H_Y(u)}\dfrac{1}{|L_{i,z}(\tilde F(z))|} + d\log ||\tilde f(z)|| \\
\label{3.7}
\begin{split}
&+\dfrac{(2k+1)\Delta}{u}\sum_{\underset{1\le i\le n_0}{1\le j\le k+1}}\log\dfrac{||\tilde f(z)||^d||P_{i_0,j}||}{|P_{i_0,j}(\tilde f)(z)|}+O(1/u).
\end{split}
\end{align}
Since $P_{i_0,1},\ldots,P_{i_0,k+1}$ are in general with respect to $Y$, By Lemma \ref{2.2}, we have
\begin{align}\label{3.8}
e_Y({\bf c}_z)\ge (c_{i_0,1,z}+\cdots +c_{i_0,k+1,z})\cdot\Delta =\left (\sum_{1\le j\le k}\log\dfrac{||\tilde f(z)||^d||P_{i_0,j}||}{|P_{i_0,j}(\tilde f)(z)|}\right )\cdot\Delta.
\end{align}
Then, from (\ref{3.4}), (\ref{3.7}) and (\ref{3.8}) we have
\begin{align}\label{3.9}
\begin{split}
\dfrac{1}{N-k+1}\log \prod_{i=1}^q\dfrac{||\tilde f (z)||^d}{|Q_i(\tilde f)(z)|}\le & \dfrac{k+1}{uH_Y(u)}\log\prod_{i=1}^{H_Y(u)}\dfrac{1}{|L_{i,z}(\tilde F(z))|}+d(k+1)\log ||\tilde f(z)||\\
&+\dfrac{(2k+1)(k+1)\Delta}{u}\sum_{\underset{1\le i\le n_0}{1\le j\le k+1}}\log\dfrac{||\tilde f(z)||^d||P_{i,j}||}{|P_{i,j}(\tilde f)(z)|}+O(1),
\end{split}
\end{align}
where the term $O(1)$ does not depend on $z$. 

Set $p=N-k+1, \ m_0=(2k+1)(k+1)\Delta$ and $b=\dfrac{k+1}{uH_Y(u)}$. Then from above inequality, we get
\begin{align}\label{3.10}
\begin{split}
\log \dfrac{||\tilde f (z)||^{\frac{1}{p}dq-d(k+1)-\frac{dm_0l}{u}}.|W^\alpha(\tilde{F}(z))|^b}{\prod_{i=1}^q|Q_i(\tilde f)(z)|^{\frac{1}{p}}}&- \dfrac{m_0}{u}\sum_{\underset{1\le i\le n_0}{1\le j\le k+1}}\log\dfrac{||P_{i,j}||}{|P_{i,j}(\tilde f)(z)|}\\
&\le b \log\dfrac{|W^\alpha(\tilde{F}(z))|}{\prod_{i=1}^{H_Y(u)}|L_{i,z}(\tilde F(z))|}+O(1).
\end{split}
\end{align}
$$\hspace{5cm}$$
Here we note that $L_{i,z}$ depends on $i$ and $z$, but the number of these linear forms is finite. We denote by $\mathcal L$ the set of all $L_{i,z}$ occuring in the above inequalities.

Then there exists a positive constant $K_0$ such that
\begin{align*}
\dfrac{||\tilde f (z)||^{\frac{1}{p}dq-d(k+1)-\frac{dm_0l}{u}}.|W^\alpha(\tilde{F}(z))|^b . \prod_{\underset{1\le i\le n_0}{1\le j\le k+1}}|P_{i,j}(\tilde f)(z)|^{\frac{m_0}{u}}}{\prod_{i=1}^q|Q_i(\tilde f)(z)|^{\frac{1}{p}}}\le K_0^b.S_{\mathcal J}^b,
\end{align*}
where $ S_{\mathcal J}=\dfrac{|W^\alpha(\tilde{F}(z))|}{\prod_{L\in\mathcal J}|L(\tilde F(z))|}$ for some $\mathcal J \subset \mathcal{L}$ with $\# \mathcal{J} = H_Y(u)$ so that $\{L\in\mathcal J\}$ is linearly independent.

We now estimate the quantity $\nu_{W(\tilde F)}(r)$. We consider a point $z\in\B^m(R_0)$ outside the indeterminacy locus of $f$. We see that $\nu^0_{Q_i(f)}(z)=0$ for all $i\ge N+1$, since $\{Q_1,\ldots,Q_q\}$ is in $N$-subgeneral position in $V$. We set $c_{i,j}=\max\{0,\nu^0_{P_{i,j}}(z)-n_u\}$ and 
$${\bf c}=(c_{1,1},\ldots,c_{1,k+1},\ldots,c_{n_0,1},\ldots,c_{n_0,k+1})\in\mathbb Z^l_{\ge 0}.$$
Then there are 
$${\bf a}_i=(a_{i,1,1},\ldots,a_{i,1,k+1},\ldots,a_{i,n_0,1},\ldots,a_{i,n_0,k+1}),a_{i,j,s}\in\{1,\ldots,l_u\}$$
such that ${\bf y}^{{\bf a}_1},\ldots,{\bf y}^{{\bf a}_{H_Y(u)}}$ is a basic of $Y_u$ and
$$ S_Y(m,{\bf c})=\sum_{i=1}^{H_Y(u)}{\bf a}_i\cdot{\bf c}.$$
Similarly as above, we write ${\bf y}^{{\bf a}_i}=L_i(v_1,\ldots,v_{H_Y(u)})$, where $L_1,\ldots,L_{H_Y(u)}$ are independent linear forms in variables $y_{i,j}\ (1\le i\le n_0,1\le j\le k+1)$. By the property of the general Wronskian, we see that 
$$W(\tilde F)=cW(L_1(\tilde F),\ldots,L_{H_Y(u)}(\tilde F)),$$
where $c$ is a nonzero constant. This yields that
$$ \nu^0_{W(\tilde F)}(z)=\nu^0_{W(L_1(\tilde F),\ldots,L_{H_Y(u)}(\tilde F))}(z)\ge\sum_{i=1}^{H_Y(u)}\max\{0,\nu^0_{L_i(\tilde F)}(z)-n_u\}$$
We also easily see that $ \nu^0_{L_i(\tilde F)}(z)=\sum_{\underset{1\le j\le n_0}{1\le s\le k+1}}a_{i,j,s}\nu^0_{P_{j,s}(\tilde f)}(z), $ and hence
$$ \max\{0,\nu^0_{L_i(\tilde F)}(z)-n_u\}\ge\sum_{i=1}^{H_Y(u)}a_{i,j,s}c_{j,s}={{\bf a}_i}\cdot{\bf c}. $$
Thus, we have
\begin{align}\label{3.11}
 \nu^0_{W(\tilde F)}(z)\ge\sum_{i=1}^{H_Y(u)}{{\bf a}_i}\cdot{\bf c}=S_Y(u,{\bf c}).
\end{align}
Since $P_{1,1},\ldots,P_{1,k+1}$ are in general position, then by Lemma \ref{2.2} we have
$$ e_Y({\bf c})\ge \Delta\cdot\sum_{j=1}^{k+1}c_{1,j}=\Delta\cdot\sum_{j=1}^{k+1}\max\{0,\nu^0_{P_{1,j}(f)}(z)-n_u\}. $$
On the other hand, by Theorem \ref{2.1} we have 
\begin{align*}
 S_Y(u,{\bf c}) &\ge\dfrac{uH_Y(u)}{(k+1)\Delta}e_Y({\bf c})-(2k+1)\Delta H_Y(u)\max_{\underset{1\le i\le n_0}{1\le j\le k+1}}c_{i,j}\\
&\ge \dfrac{uH_Y(u)}{k+1}\sum_{j=1}^{k+1}\max\{0,\nu^0_{P_{1,j}(\tilde f)}(z)-n_u\}-(2k+1)\Delta H_Y(u)\max_{\underset{1\le i\le n_0}{1\le j\le k+1}}\nu^0_{P_{i,j}(\tilde f)}(z).
\end{align*}
Combining this inequality and (\ref{3.11}), we have
\begin{align}\label{3.12}
\begin{split}
\dfrac{k+1}{duH_Y(u)}\nu^0_{W(\tilde F)}(z)\ge&\dfrac{1}{d}\sum_{j=1}^{k+1}\max\{0,\nu^0_{P_{1,j}(\tilde f)}(z)-n_u\}\\
&-\dfrac{(2k+1)(k+1)\Delta}{du}\max_{\underset{1\le i\le n_0}{1\le j\le k+1}}\nu^0_{P_{i,j}(\tilde f)}(z).
\end{split}
\end{align}
Also it is easy to see that $\nu^0_{P_{1,j}(\tilde f)}(z)\ge \nu^0_{Q_{N-k+j}(\tilde f)}(z)$ for all $2\le j\le k+1$. Therefore, we have
\begin{align*}
(N-k+1)&\sum_{j=1}^{k+1}\max\{0,\nu^0_{P_{1,j}(\tilde f)}(z)-n_u\}\\
&\ge (N-k+1)\max\{0,\nu^0_{Q_{1}(\tilde f)}(z)-n_u\}+\sum_{j=2}^{k+1}\max\{0,\nu^0_{Q_{N-k+j}(\tilde f)}(z)-n_u\}\\
&\ge\sum_{i=1}^{N+1}\max\{0,\nu^0_{Q_i(\tilde f)}(z)-n_u\}=\sum_{i=1}^{q}\max\{0,\nu^0_{Q_i(\tilde f)}(z)-n_u\}.
\end{align*}
Combining this inequality and (\ref{3.12}), we have
\begin{align*}
\dfrac{(N-k+1)(k+1)}{duH_Y(u)}&\nu^0_{W(\tilde F)}(z)\ge\dfrac{1}{d}\sum_{i=1}^{q}\max\{0,\nu^0_{Q_i(\tilde f)}(z)-n_u\}\\
&-\dfrac{(N-k+1)(2k+1)(k+1)\Delta}{du}\max_{\underset{1\le i\le n_0}{1\le j\le k+1}}\nu^0_{P_{i,j}(\tilde f)}(z)\\
\ge& \dfrac{1}{d}\sum_{i=1}^{q}(\nu^0_{Q_i(\tilde f)}(z)-\min\{\nu^0_{Q_i(\tilde f)}(z),n_u\})\\
&-\dfrac{(N-k+1)(2k+1)(k+1)\Delta}{du}\max_{\underset{1\le i\le n_0}{1\le j\le k+1}}\nu^0_{P_{i,j}(\tilde f)}(z).
\end{align*}

Therefore,
\begin{align}\label{3.13}
\dfrac{1}{p}\sum_{i=1}^{q}\nu^0_{Q_i(\tilde f)}(z)-b\nu^0_{W(\tilde F)}(z)\le \dfrac{1}{p}\sum_{i=1}^{q}\min\{\nu^0_{Q_i(\tilde f)}(z),n_u\}+\dfrac{m_0}{u}\max_{\underset{1\le i\le n_0}{1\le j\le k+1}}\nu^0_{P_{i,j}(\tilde f)}(z).
\end{align}

Assume that 
$$ \rho\Omega_f+\dfrac{\sqrt{-1}}{2\pi}\partial\bar\partial\log h^2\ge \ric\omega.$$
We now suppose that
$$ \sum_{j=1}^q\delta^{[H_Y(u)-1]}_f(Q_j)> p(k+1)+\dfrac{pm_0l}{u}+\dfrac{\rho \epsilon H_Y(u)(H_Y(u)-1)}{d}.$$
Then, for each $j\in\{1,\ldots ,q\},$ there exist constants $\eta_j>0$ and continuous plurisubharmonic function $\tilde u_j$ such that 
$e^{\tilde u_j}|\varphi_j|\le ||\tilde f||^{d\eta_j},$ where $\varphi_j$ is a holomorphic function with $\nu_{\varphi_j}=\min\{n_u,f^*Q_j\}$ and
$$ q-\sum_{j=1}^q\eta_j>  p(k+1)+\dfrac{pm_0l}{u}+\dfrac{\rho \epsilon p H_Y(u)(H_Y(u)-1)}{d}.$$
Put $u_j=\tilde u_j+\log |\varphi_j|$, then $u_j$ is a plurisubharmonic and
$$ e^{u_j}\le ||\tilde f||^{d\eta_j},\ j=1,\ldots ,q. $$
Let
$$v (z)=\log\left |(z^{\alpha_0+\cdots+\alpha_{n_u}})^b\dfrac{(W^{\alpha}(\tilde F(z)))^b}{(\prod_{i=1}^{q}Q_i(\tilde f)(z))^{\frac{1}{p}}}\right |+\dfrac{1}{p}\sum_{j=1}^q u_j(z)+\dfrac{m_0}{u}\sum_{\underset{1\le i\le n_0}{1\le j\le k+1}}\log|P_{i,j}(\tilde f)(z)|.$$
Therefore, we have the following current inequality
\begin{align*}
2\ddc[v]&\ge b[\nu_{W^{\alpha}(\tilde F)}]-\dfrac{1}{p}\sum_{j=1}^q[\nu_{Q_i(\tilde f)}]+\dfrac{1}{p}\sum_{j=1}^q2\ddc[u_j]+\dfrac{m_0}{u}\sum_{\underset{1\le i\le n_0}{1\le j\le k+1}}[\nu_{P_{i,j}(\tilde f)}(z)]\\
&\ge b[\nu_{W^{\alpha}(\tilde F)}]-\dfrac{1}{p}\sum_{j=1}^q[\nu_{Q_i(\tilde f)}]+\dfrac{1}{p}\sum_{j=1}^q[\min\{n_u,\nu_{Q_i(\tilde f)}\}]+\dfrac{m_0}{u}\max_{\underset{1\le i\le n_0}{1\le j\le k+1}}[\nu_{P_{i,j}(\tilde f)}(z)]\ge 0.
\end{align*}
This implies that $v$ is a plurisubharmonic function on $\B^m(R_0)$.

On the other hand, by the growth condition of $f$, there exists a continuous plurisubharmonic function $\omega\not\equiv\infty$ on $\B^m(R_0)$ such that
\begin{align*}
e^\omega {\rm d}V\le ||\tilde f||^{2\rho}v_m
\end{align*}
Set
$$t=\dfrac{2\rho}{\dfrac{1}{p}d\left(q-p(k+1)-\dfrac{pm_0l}{u}-\sum_{j=1}^q\eta_j\right)}>0$$ 
and 
$$\lambda (z)=(z^{\alpha_0+\cdots +\alpha_{n_u}})^b\dfrac{\left (W^{\alpha}(\tilde F(z))\right )^b\cdot \prod_{\underset{1\le i\le n_0}{1\le j\le k+1}}|P_{i,j}(\tilde f)(z)|^{\frac{m_0}{u}}}{Q_1^{\frac{1}{p}}(\tilde f)(z)\cdots Q_q^{\frac{1}{p}}(\tilde f)(z)}.$$ 

We choose $u$ the positive integer such that $\epsilon-\dfrac{pm_0l}{u}>0$ and $u\le pm_0l\epsilon^{-1}+1.$
Then, we see that
$$ \dfrac{H_Y(u)(H_Y(u)-1)b}{2}t< \dfrac{H_Y(u)(H_Y(u)-1)\dfrac{\epsilon}{p}}{2}\cdot\dfrac{2\rho}{\dfrac{1}{p}d\dfrac{\rho \epsilon H_Y(u)(H_Y(u)-1)}{d}}=1,$$
and the function $\zeta=\omega+ tv$ is plurisubharmonic on the K\"{a}hler manifold $M$. Choose a position number $\delta$ such that $0<\dfrac{H_Y(u)(H_Y(u)-1)b}{2}t<\delta<1.$
Then, we have
\begin{align}\label{3.14}
\begin{split}
e^\zeta dV&=e^{\omega +tv}dV\le e^{tv}||\tilde f||^{2\rho}v_m=|\lambda|^{t}(\prod_{j=1}^qe^{t\frac{1}{p}u_j})||\tilde f||^{2\rho}v_m\\
&\le|\lambda|^t ||\tilde f||^{2\rho+\sum_{j=1}^q\frac{1}{p}dt\eta_j}v_m=|\lambda|^t ||\tilde f||^{t\frac{1}{p}d(q-p(k+1)-\frac{pm_0l}{u})}v_m.
\end{split}
\end{align}

(a) We first consider the case where $R_0< \infty$ and $ \lim\limits_{r\rightarrow R_0}\sup\dfrac{T_f(r,r_0)}{\log 1/(R_0-r)}<\infty.$

It suffices for us to proof the Theorem in the case where $\mathbb{B}^m(R_0)=\mathbb{B}^m(1).$

Integrating both sides of (\ref{3.14}) over $\mathbb{B}^m(1),$  we have
\begin{align}\label{3.15}
\begin{split}
\int_{\B^m(1)}e^\zeta dV&\le \int_{\B^m(1)}|\lambda|^t ||\tilde f||^{t(\frac{1}{p}dq-d(k+1)-\frac{dm_0l}{u})}v_m.\\
&=2m\int_0^1r^{2m-1}\left (\int_{S(r)}\bigl (|\lambda| ||\tilde f||^{\frac{1}{p}dq-d(k+1)-\frac{dm_0l}{u}}\bigl )^t\sigma_m\right )dr\\
&\le 2m\int_0^1r^{2m-1}\left (\int_{S(r)}\sum_{\mathcal J}\bigl |(z^{\alpha_0+\cdots +\alpha_{n_u}})K_0S_{\mathcal J}\bigl |^{bt}\sigma_m\right )dr,
\end{split}
\end{align}
where the summation is taken over all $\mathcal J\subset\mathcal L$ with $\sharp J=H_Y(u)$ and $\{L\in\mathcal J\}$ is linearly independent.

We note that $(\sum_{i=0}^{n_u}|\alpha_i|)bt\le \dfrac{H_Y(u)(H_Y(u)-1)b}{2}t<\delta<1$. Then by Proposition \ref{pro2.2} there exists a positive constant $K_1$ such that, for every $0<r_0<r<r'<1,$ we have
\begin{align*}
\int_{S(r)}\left |(z^{\alpha_0+\cdots +\alpha_{n_u}})K_0S_{\mathcal J}(z)\right |^{bt}\sigma_m\le K_1\left (\dfrac{r'^{2m-1}}{r'-r}dT_f(r',r_0)\right )^{\delta}.
\end{align*}
Choosing $r'=r+\dfrac{1-r}{eT_f(r,r_0)}$, we get
$$ T_f(r',r_0)\le 2T_f(r,r_0)$$
outside a subset $E\subset [0,1]$ with $\int_E\dfrac{dr}{1-r}<+\infty$. Hence, the above inequality implies that
\begin{align*}
\int_{S(r)}\left |(z^{\alpha_1+\cdots +\alpha_{n_u}})K_0S_J(z)\right |^{bt}\sigma_m\le \dfrac{K}{(1-r)^\delta}\left (\log\dfrac{1}{1-r}\right )^{\delta}
\end{align*}
for all $r$ outside $E$, where $K$ is a some positive constant. By choosing $K$ large enough, we may assume that the above inequality holds for all $r\in (0;1)$.
Then, the inequality (\ref{3.15}) yields that
\begin{align*}
\int_{\B^m(1)}e^\zeta dV&\le 2m\int_0^1r^{2m-1}\dfrac{K}{(1-r)^\delta}\left (\log\dfrac{1}{1-r}\right )^{\delta}dr< +\infty
\end{align*}
This contradicts the results of S. T. Yau  and L. Karp (see \cite{K82,Y76}). 

Hence, we must have
$$\sum_{j=1}^q\delta^{[H_Y(u)-1]}_{f}(Q_j)\le p(k+1)+\dfrac{pm_0 l}{u}+\dfrac{\rho \epsilon H_Y(u)(H_Y(u)-1)}{d}.$$
Since $\epsilon -\dfrac{pm_0l}{u}>0$, the above inequality implies that
$$\sum_{j=1}^q\delta^{[H_Y(u)-1]}_{f}(Q_j)\le (N-k+1)(k+1)+\epsilon+\dfrac{\rho\epsilon H_Y(u)(H_Y(u)-1)}{d}.  $$
The theorem is proved in this case.

We note that $\deg Y=\Delta\le d^k\deg (V)$ and it is enough for us to consider the case where $\epsilon <q$. Then the number $n_u$ is estimated as follows
\begin{align*}
n_u&=H_Y(u)-1\le\Delta \binom{k+u}{k}\le d^k\deg (V)e^{k}\left(1+\dfrac{u}{k}\right)^{k}\\
&<d^k\deg (V)e^k\left(1+\dfrac{p(2k+1)(k+1)\Delta l\epsilon^{-1}+1}{k}\right)^k\\
&<d^k\deg (V)e^k\left(p(2k+4)\Delta l \epsilon^{-1}\right)^k\\
&<d^k\deg (V)e^k\left(p(2k+4)d^k\deg (V) l \epsilon^{-1}\right)^k.
\end{align*} 
This implies that
$$n_u\le \left[\deg (V)^{k+1}e^kd^{k^2+k}(N-k+1)^k(2k+4)^kl^k\epsilon^{-k}\right]=M_0-1.$$

(b) We now consider the remaining case where $ \lim\limits_{r\rightarrow R_0}\sup\dfrac{T(r,r_0)}{\log 1/(R_0-r)}= \infty .$

Repeating the argument in the proof of Theorem \ref{1.1}, we only need to prove the following theorem.
\begin{theorem}
With the assumption of Theorem {1.1}. Then, we have
$$(q-p(k+1)-\epsilon)T_f(r,r_0)\le \sum_{i=1}^{q}\dfrac{1}{d}N^{[M_0-1]}_{Q_i(\tilde f)}(r)+S(r),$$
where $S(r)$ is evaluated as follows:

(i) In the case $R_0<\infty,$ $$S(r)\le K(\log^+\dfrac{1}{R_0-r}+\log^+T_f(r,r_0))$$
for all $0<r_0<r<R_0$ outside a set $E\subset [0,R_0]$ with $\int_E\dfrac{dt}{R_0-t}<\infty$ and $K$ is a positive constant.

(ii) In the case  $R_0=\infty,$ $$S(r)\le K(\log r+\log^+T_f(r,r_0))$$
for all $0<r_0<r<\infty$ outside a set $E'\subset [0,\infty]$ with $\int_{E'} dt<\infty$ and $K$ is a positive constant.
\end{theorem}
	
\noindent\textit{Proof.} Repeating the above argument, we have
\begin{align*}
\int\limits_{S(r)}&\left|(z^{\alpha_0+\cdots +\alpha_{n_u}})^b\dfrac{||\tilde f(z)||^{\frac{1}{p}qd-d(k+1)-\frac{dm_0l}{u}}|W^{\alpha}(\tilde F)(z)|^b\cdot \prod_{\underset{1\le i\le n_0}{1\le j\le k+1}}|P_{i,j}(\tilde f)(z)|^{\frac{m_0}{u}}}{\prod_{i=1}^{q}|Q_i(\tilde f)(z)|^{\frac{1}{p}}}\right|^{t}\sigma_m\\
&\le K_1\left(\dfrac{R^{2m-1}}{R-r}dT_f(R,r_0)\right)^{\delta}.
\end{align*}
for every $0<r_0<r<R<R_0$. Using the concativity of the logarithmic function, we have
\begin{align*}
\begin{split}
b&\int_{S(r)}\log |(z^{\alpha_0+\cdots +\alpha_{n_u}})|\sigma_m+\left(\dfrac{1}{p}qd-d(k+1)-\dfrac{dm_0l}{u}\right)\int_{S(r)}\log ||\tilde f||\sigma_m\\
&+b\int_{S(r)}\log |W^{\alpha}(\tilde F)|\sigma_m+\dfrac{m_0}{u}\sum_{\underset{1\le i\le n_0}{1\le j\le k+1}} \int_{S(r)}\log |P_{i,j}(\tilde f)|\sigma_m\\
&-\dfrac{1}{p}\sum_{j=1}^q \int_{S(r)}\log |Q_j(\tilde f)|\sigma_m\le K\left(\log^+\dfrac{R}{R-r}+\log^+T_f(R,r_0)\right)
\end{split}
\end{align*}
for some positive constant $K$. By Jensen's formula, this inequality implies that
\begin{align}\label{3.16}
\begin{split}
\biggl(\dfrac{1}{p}qd&-d(k+1)-\dfrac{dm_0l}{u}\biggl)T_f(r,r_0)+bN_{W^{\alpha}(\tilde F)}(r)+\dfrac{m_0}{u}\sum_{\underset{1\le i\le n_0}{1\le j\le k+1}}N_{P_{i,j}(\tilde f)}(r)\\
&-\dfrac{1}{p}\sum_{i=1}^qN_{Q_i(\tilde f)}(r)\le K\left(\log^+\dfrac{R}{R-r}+\log^+T_f(R,r_0)\right)+O(1).
\end{split}
\end{align}
From (\ref{3.12}), we have 
$$ \dfrac{1}{p}\sum_{i=1}^qN_{Q_i(\tilde f)}(r)-bN_{W^{\alpha}(\tilde F)}(r)-\dfrac{m_0}{u}\sum_{\underset{1\le i\le n_0}{1\le j\le k+1}}N_{P_{i,j}(\tilde f)}(r)\le\dfrac{1}{p}\sum_{i=1}^qN^{[M_0-1]}_{Q_i(\tilde f)}(r).$$
Combining this estimate and (\ref{3.16}), we get
\begin{align}\label{new}
\begin{split}
\left(q-p(k+1)-\epsilon\right)T_f(r,r_0)\le&\sum_{i=1}^{q}\dfrac{1}{d}N^{[M_0-1]}_{Q_i(\tilde f)}(r)\\
&+K\left(\log^+\dfrac{R}{R-r}+\log^+T_f(R,r_0)\right)+O(1).
\end{split}
\end{align}
Choosing $R=r+\dfrac{R_0-r}{eT_f(r,r_0)}$ if $R_0<\infty$ and $R=r+\dfrac{1}{T_f(r,r_0)}$ if $R_0=\infty$, we see that
$$ T_f\left(r+\dfrac{R_0-r}{eT_f(r,r_0)},r_0\right)\le 2T_f(r,r_0)$$
outside a subset $E\subset [0,R_0)$ with $\int_E\dfrac{dr}{R_0-r}<+\infty$ in the case $R_0<\infty$ and  
$$ T_f\left(r+\dfrac{1}{T_f(r,r_0)},r_0\right)\le 2T_f(r,r_0)$$
outside a subset $E'\subset [0,\infty)$ with $\int_{E'}dr<\infty$ in the case $R_0=\infty$.
Thus, from (\ref{new}) we have
$$ (q-p(k+1)-\epsilon) T_f(r,r_0)\le\sum_{i=1}^{q}\dfrac{1}{d}N^{[M_0-1]}_{Q_i(\tilde f)}(r)+S(r). $$
This implies that
$$\sum_{j=1}^q\delta^{[M_0-1]}_{f}(Q_j)\le\sum_{j=1}^q\delta^{[M_0-1]}_{f,*}(Q_j)\le p(k+1)+\epsilon.$$
The theorem is proved in this case.
\end{proof}

\section{Value distribution of the Gauss map of a complete regular submanifold of $\C^m$}

Let $M$ be a connected complex manifold of dimension $m$. Let 
$$f = (f_1,\ldots , f_n) : M \rightarrow \C^n$$
be a regular submanifold of $\C^n$; namely, $f$ be a holomorphic map of $M$ into $\C^n$ such that $\rank({\rm d}_pf) = \dim M$ for every point 
$p\in M.$ We assign each point $p\in M$ to the tangent space $T_p(M)$ of $M$ at $p$ which may be considered as an $m$-dimensional linear subspace of $T_{f(p)}(\C^n)$. Also, each tangent space $T_p(\C^n)$ can be identified with $T_0(\C^n)= \C^n$ by a parallel translation. Hence, each $T_p(M)$ is corresponded to a point $G(p)$ in the complex Grassmannian manifold $G(m,n)$ of all $m$-dimensional linear subspaces of $\C^n$.

\begin{definition}
	The map $G : p\in M \mapsto G(p)\in G(m,n)$ is called the Gauss map of the map $f : M \rightarrow \C^n$. 
\end{definition}
The space $G(m,n)$ is canonically embedded in $\P^N(\C)=\P(\bigwedge^m\C^n)$, where $N =\binom{m}{n}-1$. Then we may identify the Gauss map $G$ with a holomorphic mapping of $M$ into $\P^N(\C)$ given as follows: taking holomorphic local coordinates $(z_1,\ldots ,z_n)$ defined on an open set $U$, we consider the map 
$$\bigwedge := D_1f\wedge\cdots\wedge D_nf: U\rightarrow\bigwedge^m\C^n\setminus\{0\},$$
where 
$$D_if = (\dfrac{\partial f_1}{\partial z_i},\ldots ,\dfrac{\partial f_n}{\partial z_i}).$$
Then, locally we have
$$G = \pi\circ\bigwedge,$$
where $\pi : \C^{N+1} \setminus\{0\}\rightarrow\P^N(\C)$ is the canonical projection map. A regular submanifold $M$ of $\C^m$ is considered 
as a K\"{a}hler manifold with the metric $\omega$ induced from the standard flat metric on $\C^m$. We denote by ${\rm d}V$ the volume form on $M$. For arbitrarily holomorphic coordinates $z_1,\ldots ,z_m,$ we see that
$${\rm d}V =|\bigwedge|^2\left (\frac{\sqrt{-1}}{2}\right )^m{\rm d}z_1\wedge {\rm d}\bar{z_1}\wedge\cdots\wedge {\rm d}z_m\wedge {\rm d}z_m,$$ 
where
$$|\bigwedge|^2=\sum_{1\le i_1<\cdots <i_m\le n}\dfrac{\partial(f_{i_1},\ldots,f_{i_m})}{\partial(z_1,\ldots,z_m)}^2.$$
Therefore, for a regular submanifold $f : M \rightarrow \C^m$, the Gauss map $G : M \rightarrow \P^N(\C)$ satisfies the following growth condition
$$\Omega_G + \ddc \log h^2 = \ddc \log |\bigwedge|^2 = \ric (\omega),$$
where $h = 1$.
Then Corollary \ref{1.2} immediately gives us the following.
\begin{theorem}\label{6.1.}
	Let $M$ be a complex manifold of dimension $m$ such that the universal covering of $M$ is biholomorphic to a ball $\B^m(R_0)\ (0< R_0\le +\infty)$ 
	in $\C^m.$ Let $f :M\rightarrow\C^n$ be a complete regular submanifold and $G: M\rightarrow \P^N(\C)$ be the Gauss map, where $N=\binom{m}{n}-1$. Let $Q_1,\ldots ,Q_q$ be $q$ hypersurfaces of degree $d_j \ (1\leq j\leq q)$ in general position in $\P^N(\C)$. Let $d$ be the least common multiple of $d_i$'s, i.e., $d =l.c.m.\{d_1,\ldots ,d_q\}$. Then, for every $\epsilon>0$ we have
	$$\sum_{i=1}^q\delta_G^{[M_0-1]}(Q_i) \le \left(\dfrac{N}{2}+1\right)^2+1+\frac{\rho M_0(M_0-1)}{d},$$
	where $M_0= \left[d^{N^2+N}e^N(2N+4)^Nl^N\epsilon^{-N}+1\right]$ and $l=(N+1)q!$.
\end{theorem}

\vskip0.2cm
{\footnotesize 
\noindent
{\sc Si Duc Quang}\\
$^1$ Department of Mathematics, Hanoi National University of Education,\\
 136-Xuan Thuy, Cau Giay, Hanoi, Vietnam.\\
$^2$ Thang Long Institute of Mathematics and Applied Sciences,\\
 Nghiem Xuan Yem, Hoang Mai, HaNoi, Vietnam.\\
\textit{E-mail}: quangsd@hnue.edu.vn}

\vskip0.2cm
{\footnotesize 
\noindent
{\sc Quynh Ngoc Le}\\
Faculty of Education, An Giang University, Vietnam National University Ho Chi Minh City,\\
Dong Xuyen, Long Xuyen, An Giang, Vietnam\\
\textit{E-mail}: lnquynh@agu.edu.vn}

\vskip0.2cm
{\footnotesize 
\noindent
{\sc Nguyen Thi Nhung}\\
Department of Mathematics, Thang Long University,\\
Nghiem Xuan Yem, Hoang Mai, HaNoi, Vietnam.\\
\textit{E-mail}: hoangnhung227@gmail.com}


\begin{thebibliography}{99}
\bibitem{CRY01} Z. Chen, M. Ru and Q. Yan, \textit{The truncated second main theorem and uniqueness theorems}, Sci. China Math. \textbf{53} (2010), 605--616.

\bibitem{CRY02} Z. Chen, M. Ru and Q. Yan, \textit{The degenerated second main theorem and Schmidt's subspace theorem,} Sci. China Math. \textbf{55} (2012), 1367--1380.

\bibitem{CZ01} P. Corvaja and U. Zannier, \textit{On a general Thue's equation,} Amer. J. Math. \textbf{126} (2004), 1033--1055.

\bibitem{EF01} J. Evertse and R. Ferretti, \textit{Diophantine inequalities on projective variety,} Internat. Math. Res. Notices \textbf{25} (2002), 1295--1330.

\bibitem{EF02} J. Evertse and R. Ferretti, \textit{A generalization of the subspace theorem with polynomials of higher degree,} Developments in Mathematics \textbf{16}, 175--198, Springer-Verlag, New York (2008).

\bibitem{F85} H. Fujimoto, \textit{Non-integrated defect relation for meromorphic maps of complete K\"{a}hler manifolds into $\P^{N_1}(\C)\times\cdots\times\P^{N_k}(\C)$}, Japan. J. Math. \textbf{11} (1985), no 2, 233--264.

\bibitem{K82} L. Karp, \textit{Subharmonic functions on real and complex manifolds}, Math. Z. \textbf{179} (1982), 535--554. 

\bibitem{No05} J. Noguchi, \textit{A note on entire pseudo-holomorphic curves and the proof of Cartan-Nochka's theorem}, Kodai Math. J. \textbf{28} (2005), 336--346.

\bibitem{Q16-1} S. D. Quang, \textit{Degeneracy second main theorems for meromorphic mappings into projective varieties with hypersurfaces}, Trans. Amer. Math. Soc. \textbf{371} (2019), no. 4, 2431--2453.

\bibitem{Q16-2} S. D. Quang, N. T. Q. Phuong and N. T. Nhung, \textit{Non-integrated defect relation for meromorphic maps from a Kahler manifold intersecting hypersurfaces in subgeneral position of $\P^n(\C)$},  J. Math. Anal. Appl. 452 (2017), 1434--1452.

\bibitem{R} M.Ru, \textit{Holomorphic curves into algebraic varieties,} Ann. Math. \textbf{169} (2009), 255--267.

\bibitem{RS} M. Ru and S. Sogome, \textit{Non-integrated defect relation for meromorphic maps of complete K\"{a}hler manifolds into $\P^n(\C)$ intersecting hypersurfaces}, Trans. Amer. Math. Soc. \textbf{364}  (2012), no. 3, 1145--1162.

\bibitem{TT} T. V. Tan and V. V. Truong, \textit{A non-integrated defect relation for meromorphic maps of complete K\"ahler manifolds into a projective variety intersecting hypersurfaces}, Bull. Sci. Math. \textbf{136} (2012), 111--126.

\bibitem{TQ} D. D. Thai and S. D. Quang, \textit{Non-integrated defect of meromorphic maps on K\"{a}hler manifolds}, Math. Z. \textbf{292} (2019), 211-–229.

\bibitem{Y13} Q. Yan, \textit{ Non-integrated defect relation and uniqueness theorem for meromorphic maps of a complete K\"ahler manifold into $P^n(\mathbb C)$}, J. Math. Anal. Appl. \textbf{398} (2013), 567--581.

\bibitem{Y76} S. T. Yau, \textit{Some function-theoretic properties of complete Riemannian manifolds and their applications to geometry}, Indiana U. Math. J. \textbf{25} (1976), 659--670.

\end{thebibliography}
\end{document}